\documentclass[final,nomarks]{dmtcs-episciences}

\usepackage[margin=1in]{geometry}		
\usepackage{leftindex}
\usepackage{comment}
\usepackage[color=yellow,textsize=scriptsize]{todonotes}
\usepackage{amsmath}	
\usepackage{fancyhdr}	
\usepackage{graphicx}		
\usepackage{cancel}					
\usepackage{amsthm}
\usepackage{enumerate}
\usepackage{hyperref}

\usepackage{xcolor}
\hypersetup{
    colorlinks,
    linkcolor={red!50!black},
    citecolor={blue!50!black},
    urlcolor={black}
}

\theoremstyle{definition}
\newtheorem{thm}{Theorem}[section]
\newtheorem{lem}[thm]{Lemma}
\newtheorem{prop}[thm]{Proposition}
\newtheorem{cor}[thm]{Corollary}

\theoremstyle{definition}
\newtheorem{defn}[thm]{Definition}
\newtheorem{exmp}[thm]{Example}

\theoremstyle{remark}

\usepackage[utf8]{inputenc}
\usepackage{subfigure}

\title{Fixed Point Homing Shuffles}
\author{Jonathan Parlett}
\affiliation{Drexel University, Philadelphia, USA}
\keywords{Combinatorics, Permutations, Top-swops, Card Shuffling}

\newcommand{\f}{\mathcal{F}} 
\newcommand{\fw}{\mathcal{F}_w} 
\newcommand{\T}{\mathcal{T}} 
\newcommand{\Tw}{\mathcal{T}_w} 
\newcommand{\M}{\mathcal{M}} 
\newcommand{\Mo}{\overline{\mathcal{M}}} 
\newcommand{\Mw}{\mathcal{M}_w} 
\newcommand{\Mow}{\Mo_w} 
\newcommand{\Mowp}{\Mo_{w'}} 
\newcommand{\mck}{\mathcal{M}_{\text{ck}}} 
\newcommand{\Ts}{\mathcal{T}_{\text{swops}}} 
\newcommand{\id}{\varepsilon} 
\newcommand{\tn}{\operatorname{tn }} 
\newcommand{\NN}{\mathbb{N}}

\newcommand{\xmto}[1]{\overset{ #1 }{ \mapsto }}

\newcommand{\tup}[1]{\left( #1 \right)}

\newcommand{\set}[1]{\left\{ #1 \right\}}
\newcommand{\cyc}[1]{\operatorname{cyc}_{#1}}
\newcommand{\oln}{\operatorname{oln}}
\newcommand{\card}[1]{\left| #1 \right|}

\begin{document}

\publicationdata{vol. 27:1, Permutation Patterns 2024}{2025}{9}{10.46298/dmtcs.14653}{2024-10-31; 2024-10-31; 2025-05-15}{2025-07-09}

\maketitle
\begin{abstract}\ \\
    We study a family of maps from $S_n \to S_n$ we call \emph{fixed point homing shuffles}. These maps generalize a few known problems such as Conway's Topswops first described by Martin Gardner (1988), and a card shuffling process studied by Gweneth McKinley (2015).
    We show that the iterates of these homing shuffles always converge, and characterize the set $U_n$ of permutations that no homing shuffle sorts.
    We also study a homing shuffle that sorts anything not in $U_n$, and find how many iterations it takes to converge in the worst case.
\end{abstract}

\section{Introduction}

\subsection{Card Games}

Given a deck of playing cards labeled $1,2,\dots,n$, we may play a round of the \textit{topswops} game as follows. If the top card of the deck is numbered $k$, we take the first $k$ cards from the top of the deck and then place them back on top in reverse order. We may repeat this process. For example, for a deck of $5$ cards arranged $43215$ (from top to bottom), we proceed as follows, underlining the cards that will be reversed on the next iteration.
\begin{align*}
\underline{342}15 \to \underline{24}315 \to \underline{4231}5 \to \underline{1}3245.
\end{align*}
We may note that any further rounds simply pick up the first card and put it back on the top. So the deck stays the same if the $1$-card is in front. The topswops game was proposed by John Conway \cite[Chapter Six]{gardner2020}\footnote{The alternative spelling ``topswaps'' is also widespread.}, and subsequently addressed by Herbert Wilf, Donald Knuth, and recently by Linda Morales, and Hal Sudborough \cite{gardner2020,TAoCP4A,morales2010}.  In particular Wilf showed that, perhaps surprisingly, the $1$-card always eventually comes to the front in at most $2^{n-1}$ iterations \cite[Chapter Six, Appendix]{gardner2020}. This bound was later improved by Knuth to the $f_{n+1}-1$ where $f_{n+1}$ is the $(n+1)$-th Fibonacci number \cite[\S 7.2.1.2, Exercise 108]{TAoCP4A}. Recently a quadratic lower bound was shown by Morales and Sudborough \cite{morales2010}. Topswops is also called the \textit{deterministic pancake problem} \cite{DWest}, as it is related to the \textit{pancake problem} where the goal is to sort a list using the minimal number of prefix reversals \cite{chitturi200918,cohen1995problem,gates1979bounds,heydari1997diameter}.

Topswops has some similarities with another card game studied by Gweneth McKinley in \cite{McK2015}. In this game, the front card $k$ is placed at position $k$ and every card lying in between simply moved one step upwards. For example, on our above deck, with the cards to be shifted underlined, the game proceeds as follows.
\[
3\underline{42}15 \to 4\underline{231}5 \to 2\underline{3}145 \to 3\underline{21}45 \to 2\underline{1}345 \to \underline{1}2345.
\]
The two games are similar in that they both involve moving the front card to its natural position in each step. The ultimate outcomes also have something in common: In both games, the $1$-card eventually lands in its natural position at the top of the deck, and the state remains the same after this point. McKinley showed that the $1$-card eventually comes to the front in at most $2^{n-1}-1$ iterations, and that this bound is tight. That is, there are decks such that it takes exactly $2^{n-1}-1$ turns for the $1$-card to come to the front. These similarities suggest a generalization we will explore in this work.

\subsection{Contributions}

In this work we study a family of games that subsume the above. We show that, for any game in our family, the $1$-card always comes to the front in at most $2^{n-1}$ iterations, generalizing the results of Wilf and McKinley. We note that, for our above example, McKinley's shuffle sorted the deck in question to its natural order $12345$, while Topswops did not. We characterize the decks that are not sortable by any game in our family, and exhibit a game that sorts all decks that are sortable. For this game, which we call \textit{max shuffle}, we show that the $1$-card comes to the front in at most $2n-3$ iterations, and that this bound is tight. During the review process, we were informed that much of our results here are present in a prior work of Joyce Pechersky \cite{Pe2023}. In particular, our Theorem \ref{f-terminates} and Theorem \ref{cardUn} appear as Theorem 2 and Theorem 4 in \cite{Pe2023}. Pechersky also studies our map $\T$ in more depth than we cover here. See Theorem 5 and Corollary 4 of \cite{Pe2023}. Our discussion of max shuffles remains a novel contribution, and we hope that our elementary presentation serves as good starting point for further research.

\section{Notations}
Let $\NN = \set{0,1,2,\ldots}$ be the set of non-negative integers.
For $n \in \NN$ we set $[n] = \set{1,2,\ldots,n}$.
For $a,b \in \NN$ we define the closed interval $[a,b] = \set{a,a+1,a+2,\ldots,b-1,b}$, and the open interval $(a,b) = \set{a+1,a+2,\ldots,b-1}$, with the half open $(a,b],[a,b)$ intervals defined similarly.
We denote the symmetric group by $S_n := \{\sigma : [n] \to [n] : \sigma \text{ is a bijection}\}$.
Its multiplication is composition: $(ab)(i) = a(b(i))$ for all $a, b \in S_n$ and $i \in [n]$.
We denote the identity in $S_n$ by $\varepsilon$.
We will use $ \cyc{i_{1},i_{2},\dots,i_{k}} $ to describe a cycle in $ S_{n} $, and in particular a transposition will be denoted $ t_{i,j} $.
We define the \emph{one-line notation} of a permutation $ w \in S_{n} $ to be the $n$-tuple $ \oln\tup{w} :=  \tup{ w\tup{ 1 },w\tup{ 2 },\dots,w\tup{ n } }$. When it is not ambiguous, we may omit the commas. For example, for $t_{1,2} \in S_3$ we may write $\oln{t_{1,2}}=\tup{213}$. Occasionally, we will identify the permutation $w$ with its one-line notation.

\section{Fixed Point Homing Shuffles}

One way to characterize the shuffles we have described so far is that they always place the front card in its natural position, and do not disturb any cards numbered higher than the front card. We take these key properties as our definition of a homing shuffle.

\begin{defn}
    \label{shufDef}
    A \emph{homing shuffle} is a map $ \f : S_{n} \to S_{n} $ that satisfies the following conditions: For any $ w \in S_{n} $, if we set $ k := w\tup{ 1 } $, then
    \begin{enumerate}[]
        \item\textbf{(a)} We have $ \f\tup{ w }\tup{ k } = k $. Thus, $ k $ is always a fixed point of $ \f\tup{ w } $.
        \item\textbf{(b)} We have $ \f\tup{ w }\tup{ i } = w\tup{ i } $ for all $ i > k $.
    \end{enumerate}
    To keep notation concise, we will set $\fw := \f\tup{w} $. Note that $\f^k_w$ will always mean $\f^k\tup{w}$, never $\tup{\f\tup{w}}^k$.
\end{defn}
We may informally describe a homing shuffle as follows. Consider a deck of cards numbered $1,2,\dots,n$. If the top card is numbered $k$, take the first $k$ cards from the top of the deck. From this pile take the $k$-card and return it to the top of the deck. For the remaining cards return them to the top of the deck in any order. Specifying the order in which you return the remaining cards to the top of the deck yields different shuffles. For example, placing the remaining cards back in reverse order yields the Topswops game, while simply placing them back in their original order yields McKinley's.

We now consider some notable examples. In what follows, let us use the notation $w \overset{\f}{ \mapsto } w'$ as shorthand for $w' = \f(w)$. Our first example is in some sense the simplest homing shuffle.

\begin{exmp}[Transposition Shuffle]
    \label{examp1}
   Consider the map
   \begin{align*}
       \T : S_{n} &\to S_{n}, \\
       w &\mapsto w \cdot t_{1, k}, 
       \qquad \text{ where } k = w\tup{1}.
   \end{align*}
   For example, if $ w \in S_{5} $ with $ \oln\tup{w} = \tup{2 3 4 1 5} $, then (underlining the cards to be swapped in the next iteration) we see that
   \begin{align*}
        \tup{\underline{2} \underline{3} 4 1 5} &\xmto{\T} \tup{\underline{3} 2 \underline{4} 1 5} \xmto{\T} \tup{\underline{4} 2 3 \underline{1} 5} \xmto{\T} \tup{\underline{1} 2 3 4 5} .
   \end{align*}
   
   To see that $\T$ is indeed a homing shuffle, we note that condition (a) of Definition \ref{shufDef} is satisfied since
   \[
       \Tw\tup{ k } = \tup{w \cdot t_{1,k}} \tup{ k } = w\tup{t_{1,k}\tup{k}} = w\tup{1} = k .
   \]
   For all $ i \not\in \set{1, k} $, we have $t_{1,k}\tup{i} = i$ and therefore
   \[
       \Tw\tup{ i } = \tup{w \cdot t_{1,k}}\tup{ i } = w\tup{t_{1,k}\tup{i}}  = w\tup{ i } .
   \]
   So, in particular, if $ i > k $ then $  \T_w\tup{ i } = w\tup{ i } $, and condition (b) of Definition \ref{shufDef} is satisfied. $\T$ may be described informally as the shuffle that swaps the the front card $k$ with the card currently in position $k$ of the deck. 
\end{exmp}

\begin{exmp}[McKinley's Shuffle]
    This shuffle was first studied by Gweneth McKinley in \cite{McK2015}.
    It is defined as the map
    \begin{align*}
        \mck : S_{n} & \to S_{n},\\	
        w & \mapsto w\cyc{1,2,\dots,k}, \qquad \text{where } k = w\tup{1} .
    \end{align*}
    For example, considering again our permutation $w \in S_5$ with $\oln{w} = \tup{2 3 4 1 5}$, and underlining the portion of the one line notation to be shifted back next iteration, we see that  
    \begin{align*}
        \tup{2 \underline{3} 4 1 5} &\xmto{\mck} \tup{3 \underline{2 4} 1 5} \xmto{\mck} \tup{2 \underline{4} 3 1 5} \xmto{\mck} \tup{4 \underline{2 3 1} 5} \xmto{\mck} \tup{2 \underline{3} 1 4 5}\\
        &\xmto{\mck} \tup{3 \underline{2 1} 4 5} \xmto{\mck} \tup{2 \underline{1} 3 4 5} \xmto{\mck} \tup{1 2 3 4 5}.
    \end{align*}
   To see that $\mck$ is a homing shuffle, we note that
   \[
       \mck\tup{ w }\tup{ k } = w \tup{ \cyc{1,2,\dots,k }\tup{ k }} = w\tup{ 1 } = k.
   \]
   And for $ i > k $ we have $ \cyc{1,2,\dots,k}\tup{ i } = i $ so that 
   \[
       \mck\tup{ w }\tup{ i } = w \tup{ \cyc{1,2,\dots,k }\tup{ i }} = w\tup{ i }.
   \]
   As we stated above, $\mck$ can be described as the shuffle you get by simply placing the remaining cards (after placing the front card $k$) back on top of the deck without disturbing their original order.
\end{exmp}

For the next example we'll introduce the reflections in $ S_{n} $.
\begin{defn}
    If $ m \in [n] $ then we define the permutation $r_m \in S_n$ by setting
    \[
        r_{m}\tup{ i }
        = \begin{cases}
            m-i+1, & \text{ if } i \le m\\
            i, & \text{ otherwise}.
        \end{cases}
    \]
    In other words, $r_m$ sends $1,2,\ldots,m$ to $m,m-1,\ldots,1$ and leaves the remaining elements of $[n]$ unchanged.
\end{defn}

\begin{exmp}[Topswops]
    The topswops game was introduced by John Conway \cite[Chapter Six]{gardner2020} (see also \cite[Aufgabe 1]{BWM2002-2} and \cite[\S 7.2.1.2, Exercise 107]{TAoCP4A}).
    It can be described by the following map:
    \begin{align*}
        \Ts : S_{n} & \to S_{n},\\
        w & \mapsto w r_{k}, \qquad \text{where } k = w\tup{1}.
    \end{align*}
    For example, if $\oln{w} = \tup{23415}$, underlining the prefix to be reversed we have
    \begin{align*}
        \tup{\underline{23}415} &\xmto{\Ts} \tup{\underline{324}15} \xmto{\Ts} \tup{\underline{4231}5} \xmto{\Ts} \tup{\underline{1}3245}.
    \end{align*}
    Now $\Ts$ is indeed a homing shuffle since 
    \[
    \Ts\tup{ w }\tup{ k } = w r_{k}\tup{ k } = w\tup{ 1 } = k
    \]
    and $ r_{k} $ fixes all $ i > k $, so that $ \Ts\tup{ w }\tup{ i } = wr_{k}\tup{ i } = w\tup{ i } $ for all $ i > k $. Described informally, $\Ts$ is the shuffle you get by picking up the first $k$ cards and placing them back on the top of the deck in reverse order.
\end{exmp}

Before we turn to an analysis of homing shuffles more generally, we first establish some simple but useful identities concerning how homing shuffles must change the deck from one iteration to the next. 
\begin{prop}
    \label{prop1}
    Let $ \f $ be a homing shuffle, $ w \in S_{n} $, and $ k := w\tup{ 1 } $. Then
    \begin{enumerate}[]
        \item\textbf{(a)} We have $\fw\tup{ [i,j] } = w\tup{ [i,j] }$ for any $ k < i \le j \le n $.
        \item\textbf{(b)} We have $\fw\tup{ [i] } = w\tup{ [i] }$ for any $ i \ge k $.
        \item\textbf{(c)} We have $\fw\tup{ [k-1] } = w\tup{ [k] } \setminus \set{k}$.
    \end{enumerate}
\end{prop}

\begin{proof}
(a) By part (b) of Definition \ref{shufDef},
    \[
    \fw\tup{ i } = w\tup{ i } \text{ for } i > k.
    \]
This implies $ \fw\tup{ [i,j] } = w\tup{ [i,j] } $ for $ k < i \le j \le n$, and (a) is proved.

(b) Let $i \geq k$. Then, $k < i+1 \leq n \leq n$. Thus, applying (a) to $i+1$ and $n$ instead of $i$ and $j$, we find $\fw\tup{ [i+1,n] } = w\tup{ [i+1,n] }$.
But both $\fw$ and $w$ are permutations, and thus respect complementation:
\begin{align*}
\fw\tup{[n] \setminus [i+1,n]} &= [n] \setminus \fw\tup{[i+1,n]}
\qquad \text{ and } \\
w\tup{[n] \setminus [i+1,n]} &= [n] \setminus w\tup{[i+1,n]}.
\end{align*}
The right hand sides here are equal, since $\fw\tup{ [i+1,n] } = w\tup{ [i+1,n] }$. Thus the left hand sides are equal, \textit{i.e}., we have $\fw\tup{[n] \setminus [i+1,n]} = w\tup{[n] \setminus [i+1,n]}$. But this is just saying $\fw\tup{[i]} = w\tup{[i]}$. This proves part (b).

(c) We have
\begin{align*}
\fw\tup{[k-1]}
&= \fw\tup{[k] \setminus \set{k}}
= \fw\tup{[k]} \setminus \set{\fw\tup{k}}
= w\tup{[k]} \setminus \set{k} ,
\end{align*}
where the last step used the facts that $ \fw\tup{[k]} =w\tup{[k]}$ (a particular case of part (b)) and $\fw\tup{k} = k$ (a consequence of Definition \ref{shufDef} (a)).
Thus part (c).
\end{proof}

\section{Termination Numbers}

The most well studied question regarding the Topswops game is concerning the longest sequence of moves.
Gardner, in \cite{gardner2020}, asks for the ``arrangement of the numbers [that] provides the longest game of topswops''.
This question has been open since at least 1988 with only upper and lower bounds being known \cite{gardner2020,TAoCP4A,morales2010}.
McKinley, however, determines the exact number $(2^{n-1}-1)$ for $\mck$ in \cite{McK2015}.
Here, we are interested in what we can say with regards to all homing shuffles.
In this section we will show that all homing shuffles terminate in a finite number of iterations.

We first note a technical point highlighted in the introduction: the game itself may continue for ever but the state no longer changes. Indeed, any further moves simply pick up the $1$-card and place it back on top of the deck.
Thus, when we speak of termination, we really mean that the state is no longer changing after some iteration. This said, we begin our discussion with a simple observation.

Given a homing shuffle $\f$ and a permutation $w \in S_n$, we have $\f\tup{w} = w$ if and only if $w\tup{1} = 1$. Indeed, if we set $k = w\tup{1}$ as in Definition \ref{shufDef}, then $k = 1$ implies $\f\tup{w} = w$ by Definition \ref{shufDef}, whereas conversely, $\f\tup{w} = w$ implies $w\tup{1} = k = \f\tup{w}\tup{k} = w\tup{k}$ and thus $1 = k$ by unapplying $w$. This leads to the following definition of termination.

\begin{defn}[Termination Numbers] \ \\ \vspace{-1pc}
    \begin{enumerate}[]
        \item\textbf{(a)} Say that a homing shuffle $ \f $ \emph{terminates} after $m$ iterations on a given permutation $w \in S_n$ if $ \f^{ m }_w\tup{ 1 }=1 $ (or, equivalently, $\f^{m+1}\tup{w} = \f^m\tup{w}$).
        \item\textbf{(b)} The \emph{termination number} of $ \f $, denoted $ \tn \f $, is the smallest $ m \in \mathbb{N} $ such that $ \f $ terminates after $m$ iterations on every $ w \in S_{n} $. In other words,
    \[
    \tn\f := \min\set{m \in \mathbb{N} \;|\; \f^{m}\tup{ w }\tup{ 1 } = 1 \text{ for all }  w \in S_{n}} .
    \]
    The existence of such an $m$ is not obvious, but follows from Theorem \ref{f-terminates} below.
    \end{enumerate}
\end{defn}

\begin{prop}[Termination is forever]
    \label{termPr}
   Fix a $ w \in S_{n} $.
   Let $m \in \mathbb{N}$ be such that $\f$ terminates after $m$ iterations on $w$.
   Then, $ \f^{q}\tup{ w } = \f^{m}\tup{ w } $ for all $ q \geq m $.
   Thus, $\f$ terminates after $q$ iterations on $w$ for all $q \geq m$.
\end{prop}
\begin{proof}
   Since $\f$ terminates after $m$ iterations on $w$, we have $\f^m\tup{w} = \f^{m+1}\tup{w}$.
   Applying $\f$ to this equality, we obtain
   $\f^{m+1}\tup{w} = \f^{m+2}\tup{w}$.
   Applying $\f$ once again, we obtain
   $\f^{m+2}\tup{w} = \f^{m+3}\tup{w}$.
   Proceeding like this indefinitely, we arrive at the chain of equalities
   $\f^m\tup{w} = \f^{m+1}\tup{w} = \f^{m+2}\tup{w} = \cdots$.
   In other words, $ \f^{q}\tup{ w } = \f^{m}\tup{ w } $ for all $ q \geq m $.
   This entails $ \f^{q}\tup{ w }\tup{1} = \f^{m}\tup{ w } \tup{1} = 1$ for all $q \geq m$, so that $\f$ terminates after $q$ iterations on $w$.
\end{proof}

To illustrate let us determine the termination number of the simplest shuffle $ \T $.

\begin{thm}
\label{T-n-1}
The transposition shuffle $\T$ terminates after $ n-1 $ iterations on each $w \in S_n$. We have
   \[
   \tn\T = n - 1.
   \]
\end{thm}
\begin{proof}
Let $ w \in S_{n} $. Write the disjoint cycle decomposition of $ w $ as follows:
\[
w = \cyc{i_{1,1},i_{1,2},\dots,i_{1,k_{1}}} \cdots \cyc{i_{m,1},i_{m,2},\dots,i_{m,k_{m}}} \cdot \cyc{1,w\tup{ 1 },w^{2}\tup{ 1 },\dots,w^{k}\tup{ 1 }} ,
\]
where we put the cycle containing $1$ at the very end (we assume that this cycle is not trivial, since otherwise $\T$ terminates immediately on $w$).
Now by the definition of $ \T\tup{ w } $ we have
\begin{align*}
    \T\tup{ w } &= w t_{1,w\tup{ 1 }}
                = \cyc{i_{1,1},i_{1,2},\dots,i_{1,k_{1}}} \cdots \cyc{i_{m,1},i_{m,2},\dots,i_{m,k_{m}}} 
                \cdot \underbrace{\cyc{1,w\tup{ 1 },w^{2}\tup{ 1 }\dots,w^{k}\tup{ 1 }} \cdot \, t_{1,w\tup{ 1 }}}_{
                =\cyc{1,w^{2}\tup{ 1 },\dots,w^{k}\tup{ 1 }}} 
                \\
                &= \cyc{i_{1,1},i_{1,2},\dots,i_{1,k_{1}}} \cdots \cyc{i_{m,1},i_{m,2},\dots,i_{m,k_{m}}} \cdot \cyc{1,w^{2}\tup{ 1 },\dots,w^{k}\tup{ 1 }}.
\end{align*}
Thus, $\T\tup{w}$ has the same disjoint decomposition as $w$, except that the cycle containing $1$ has been shortened by one element (namely, the element $w\tup{1}$ has been removed from it and is now a fixed point).
Iterating this, we see that any further application of $\T$ reduces the length of the cycle containing $1$ by $1$. After at most $n-1$ iterations, this length will become $1$, so the permutation will fix $1$, and thus $\T$ terminates after $n-1$ iterations on $w$.
Moreover, $\T$ does not terminate any earlier when $w$ is an $n$-cycle, so we conclude that $ \tn\T = n - 1 $.
\end{proof}

As we mentioned above, all homing shuffles terminate, not just $\T$. An elegant argument of Herbert S. Wilf gives a bound on the termination number of $ \Ts $ \cite[Appendix to Chapter Six]{gardner2020}, but generalizes easily to the case of any homing shuffle:

\begin{defn}[Wilf Number]
The \emph{Wilf number} of a permutation $ w \in S_{n} $ is given by
   \begin{align*}
       W\tup{w} &:= \sum_{\substack{i \in [n];\\ w\tup{ i }=i}} 2^{i-2} \in \mathbb{Q}_{\geq 0}.
   \end{align*}
We note that if $w(1) \ne 1$ then $W(w) \in \mathbb{Z}_{\ge 0}$, since all $i \in [n]$ satisfying $w\tup{i} = i$ are $\geq 2$ and thus the addends $2^{i-2}$ are positive integers.
\end{defn}

\begin{prop}
    \label{wBound}
The Wilf number is bounded above by $ 2^{n-1} $. That is,
    \[
    W\tup{ w } < 2^{n-1} \text{ for all } w \in S_{n}
    \]
\end{prop}
\begin{proof}
Let $w \in S_n$. Then,
   \begin{align*}
       W\tup{ w } &= \sum_{\substack{i \in [n];\\ w\tup{ i }=i}} 2^{i-2} \le \sum_{\substack{i \in [n]}} 2^{i-2}  = \frac{1}{2}\tup{ \sum_{i=1}^{n}2^{i-1} } = \frac{1}{2}\tup{ \sum_{i=0}^{n-1}2^{i} }\\
       &= \frac{1}{2}\tup{ 2^{n-1+1}-1 } = 2^{n-1}-\frac{1}{2} < 2^{n-1}
   \end{align*} 
\end{proof}

As the next lemma shows, any homing shuffle increases the Wilf number of a permutation $w$, as long as it has not yet terminated on $w$.
\begin{lem}
\label{lemwilf}
    Let $w \in S_n$ be a permutation with $w\tup{1} \neq 1$.
    Let $\f$ be any homing shuffle.
    Then, $W\tup{\f_w} > W\tup{w}$.
\end{lem}

\begin{proof}
    Let $k = w\tup{1}$.
    By the definition of Wilf numbers, we have
        \begin{align}
        W\tup{ \f_w } - W\tup{ w }
        &= \sum_{\substack{i \in [n];\\ \f_w\tup{ i }=i}} 2^{i-2} - \sum_{\substack{i \in [n];\\ w\tup{ i }=i}} 2^{i-2} .
        \label{pf.lemwilf.1}
        \end{align}
    However, part (2) of Definition \ref{shufDef} shows that $\f_w\tup{i} = w\tup{i}$ for all $i>k$.
    Thus, the two sums in \eqref{pf.lemwilf.1} agree in all their addends with $i > k$ (meaning that each such addend that appears in one sum must also appear in the other).
    This allows us to cancel all these addends, thus reducing both sums to $i \in [k]$.
    So \eqref{pf.lemwilf.1} becomes
        \begin{align}
        W\tup{ \f_w } - W\tup{ w }
        &= \sum_{\substack{i \in [k];\\ \f_w\tup{ i }=i}} 2^{i-2} - \sum_{\substack{i \in [k];\\ w\tup{ i }=i}} 2^{i-2} .
        \label{pf.lemwilf.2}
        \end{align}
    However, part (1) of Definition \ref{shufDef} shows that $\f_w\tup{k} = k$. Hence, the first sum in \eqref{pf.lemwilf.2} contains at least an addend for $i = k$. We can thus bound it from below:
    \begin{align*}
        \sum_{\substack{i \in [k];\\ \f_w\tup{ i }=i}} 2^{i-2} \ge 2^{k-2}.
    \end{align*}
    On the other hand, we have $w\tup{1} \neq 1$ and $w\tup{k} \neq k$ (since $k = w\tup{1} \neq 1$ entails $w\tup{k} \neq w\tup{1} = k$).
    Hence, the second sum in \eqref{pf.lemwilf.2} contains neither an addend for $i = 1$ nor an addend for $i = k$.
    We can thus bound it from above:
    \begin{align*}
        \sum_{\substack{i \in [k];\\ w\tup{ i }=i}} 2^{i-2}
        &\le \sum_{i=2}^{k-1} 2^{i-2}
        = 2^0 + 2^1 + \cdots + 2^{k-3}
        = 2^{k-2} - 1.
    \end{align*}
    Now, \eqref{pf.lemwilf.2} becomes
        \begin{align*}
        W\tup{ \f_w } - W\tup{ w }
        &= \underbrace{\sum_{\substack{i \in [k];\\ \f_w\tup{ i }=i}} 2^{i-2}}_{\ge 2^{k-2}} - \underbrace{\sum_{\substack{i \in [k];\\ w\tup{ i }=i}} 2^{i-2}}_{\le 2^{k-2} - 1}
        \ge 2^{k-2} - \tup{2^{k-2} - 1} = 1.
        \end{align*}
    Hence, $W\tup{ \f_w } \geq W\tup{ w } + 1 > W\tup{w}$.
\end{proof}

We now show a bound on the termination number of any shuffle $ \f $.
\begin{thm}
    \label{f-terminates}
    Let $ \f $ be a homing shuffle. Then, $\f$ terminates in less than $2^{n-1}$ iterations on each $w \in S_n$. In other words,
    \[
    \tn\f < 2^{n-1}
    \]
\end{thm}

\begin{proof}
    Let $w \in S_n$.
    We must show that $\f$ terminates in less than $2^{n-1}$ iterations on $w$.
Assume to the contrary that $\f$ requires at least $2^{n-1}$ iterations to terminate on $w$. That is $\f^k\tup{w} \tup{1} \neq 1$, for any $k < 2^{n-1}$.
    Hence, the Wilf numbers 
    \[
    W\tup{\f^0\tup{w}},\ W\tup{\f^1\tup{w}},\ \ldots,\ W\tup{\f^{2^{n-1}}\tup{w}}
    \] all belong to $\mathbb{Z}_{\ge 0}$. Proposition \ref{wBound} shows that all these Wilf numbers are $< 2^{n-1}$, and Lemma \ref{lemwilf} shows that they strictly increase:
    \[
    W\tup{\f^0\tup{w}} < W\tup{\f^1\tup{w}} < \cdots < W\tup{\f^{2^{n-1}}\tup{w}}.
    \]
    Hence, these Wilf numbers are $2^{n-1}+1$ distinct integers in the half-open interval $[0, 2^{n-1})$.
    But this is impossible, since this interval contains only $2^{n-1}$ integers.
    This contradiction completes our proof.
\end{proof}

\section{Sortable Permutations}
Card decks that eventually map to the identity under a given homing shuffle have also been of interest in the literature.
These played a prominent role in establishing the lower bound for Topswops in \cite{morales2010}, where the authors used a method of combining such permutations to generate ones that take more iterations to terminate. 
We may formalize this notion of a permutation mapping to the identity as follows.
\begin{defn}
    Say a homing shuffle $ \f $ \emph{sorts} $  w \in S_{n}$ if $ \f^{m}\tup{ w } = \id \in S_{n} $ for some $ m \in \mathbb{N} $. If there is no such $m$ then we say $w$ is  \emph{not sortable} or \emph{unsortable} by $\f$.
\end{defn}
We may occasionally refer to the collection of all permutations $w$ that $\f$ sorts as the \emph{sortables} of $\f$.
For a first example of characterizing the sortables of a homing shuffle, consider again the transposition shuffle $\T$.
In the proof of Theorem \ref{T-n-1}, we found that $\T$ gradually reduces the cycle of $w$ that contains $1$, while leaving all other cycles unchanged.
Hence, after $n-1$ iterations of $\T$, the cycle of $w$ that contains $1$ will have disappeared (with all its elements becoming fixed points), while all the other cycles of $w$ will still be present.
Thus, $\T^{n-1}\tup{w}$ is the identity $ \id \in S_{n} $ if and only if the permutation $w$ has no other nontrivial cycles beyond the cycle that contains $1$ -- in other words, if and only if $ w $ is a single cycle that contains $1$.
We have thus shown the following.
\begin{prop}
    \label{sortableT}
    The shuffle $ \T $ sorts $ w \in S_{n} $ if and only if $ w $ is a single cycle that contains $1$. In other words, $ \T $ sorts $ w $ if and only if
    \[
        w = \cyc{i_{1},i_{2},\dots,i_{k}} \text{ for some } i_1,i_2,\dots,i_k \in [n] \text{ with } i_1 = 1.
    \]
\end{prop}

Thus we see that the sortable permutations of $ \T $ are easy to characterize.
This is not the case for other homing shuffles.
No characterization is known for the sortables of  $ \Ts $, and McKinley provides only a partial characterization of the sortables of $ \mck $ including some necessary conditions, and a bound on their number \cite{McK2015}.
Fitting with the generality of our approach, we seek to characterize the set of permutations that are not sortable by any homing shuffle $\f$. As we will see, these permutations are related to the well-known \emph{irreducible permutations}.
\begin{defn}[irreducible/reducible permutations]
    \label{irredPerms}
    Let $ w \in S_{n} $.
    \begin{enumerate}[]
        \item\textbf{(a)} We let $ i_{w} \in [n] $ be the smallest index $k \in [n]$ such that $w\tup{ [k] } = [k]$.
        \item\textbf{(b)} We say that $ w$ is \emph{irreducible} if $i_w = n$.
        Otherwise, we say that $ w $ is \emph{reducible}.
    \end{enumerate}
\end{defn}

\begin{exmp}
The permutation $w = \tup{3,2,1,5,4} \in S_5$ is reducible with $i_w = 3$ since $w([3])=[3]$. Meanwhile, the permutation $w' = \tup{2,3,4,5,1} \in S_5$ is irreducible as $w(5)=1$ and thus $1 \notin w([i])$ for any $i < 5$. More generally, any permutation in $S_n$ with $w(n)=1$ is irreducible by the above argument.
\end{exmp}

It is well known that the number of irreducible permutations in $ S_{n} $ satisfies the following recurrence \cite{KING2006}.

\begin{prop}
    \label{prop5}
    Let $ I_{n} $ denote the set of irreducible permutations in $ S_{n} $. Then
    \[
    |I_{n}| = n! - \sum_{i=1}^{n-1}|I_{i}|\tup{ n-i }!.
    \]
\end{prop}

\begin{proof}
Since a permutation is either reducible or irreducible, it suffices to count the reducible permutations, since
\[
\card{I_n} = n! - \card{S_n \setminus I_n} = n! - \tup{\text{\# of reducible } w \in S_n}.
\]
We recall that the reducible permutations $w \in S_n$ are those $w \in S_n$ that satisfy $i_w \in [n-1]$.
Thus we fix an $i \in [n-1]$, and count the permutations $w \in S_n$ with $i_w = i$. By the definition of $i_w$ the restriction $w|_{[i]}$ must be an irreducible permutation in $S_i$ while the restriction $w|_{[i+1,n]}$ is any permutation in $S_{[i+1,n]}$.
Thus,
\begin{align*}
    \tup{\text{\# of } w \in S_n \text{ such that } i_w = i} &= \tup{\text{\# of irreducible permutations in } S_i} \cdot \card{S_{[i+1,n]}}\\
    &= \card{I_i}(n-i)!,
\end{align*}
since $\card{S_{[i+1,n]}} = (n-i)!$.

Summing this over all $i \in [n-1]$, we obtain
\[
\tup{\text{\# of reducible } w \in S_n}
= \sum_{i=1}^{n-1} \card{I_i}(n-i)!,
\]
since the reducible permutations $w \in S_n$ are those $w \in S_n$ that satisfy $i_w \in [n-1]$.
Hence,
\begin{align*}
\card{I_n} &= n! - \tup{\text{\# of reducible } w \in S_n}  = n! - \sum_{i=1}^{n-1} \card{I_i}(n-i)!.
\end{align*}
\end{proof}

\begin{defn}[The unsortables]
    \label{def5}
    Let $ i_{w} $ be as in Definition \ref{irredPerms} for any $ w \in S_{n} $. Then the \emph{unsortables} in $ S_{n} $ are the permutations that belong to the set
   \[
   U_{n} := \set{w \in S_{n} | \text{ there is } i > i_{w} \text{ such that } w\tup{ i } \ne i}.
   \]
    In other words, they are the permutations $w \in S_n$ that have at least one non-fixed point larger than $i_w$.
\end{defn}

\begin{exmp}
    The permutation $w = \tup{3,2,1,5,4} \in S_5$ is unsortable, since $i_w=3$, but $5$ and $4$ are not fixed points of $w$. The permutation $w' = \tup{3,2,1,4,5}$ is not unsortable, since $4,5$ are fixed points of $w'$.
    An irreducible permutation $w \in S_n$ is never unsortable, since there is no $i > i_w = n$. Thus, for example, the permutation $\tup{5,4,3,2,1} \in S_5$ is not unsortable because it is irreducible.
\end{exmp}

Now we show the unsortables are indeed unsortable by any homing shuffle. Intuitively this is because the upper portion of the deck $[w\tup{i_w+1},w\tup{i_w+2},\dots,w\tup{n}]$ remains undisturbed by the shuffle process (\textit{i.e}., we have $\fw\tup{i} = w(i)$ for all $i > i_w$). Thus any non-fixed point of $w$ larger than $i_w$ remains a non-fixed point of $\fw$, as the proof of the below lemmas shows.

\begin{lem}
\label{i-decreases}
If $w \in U_n$, and $\f$ is any homing shuffle, then $i_{\fw} \leq i_w$.
\end{lem}

\begin{proof}
Let $w \in U_n$, and $\f$ an arbitrary homing shuffle. Set $k := w(1)$. Then, $k = w(1) \in w([i_w]) = [i_w]$, so that $k \le i_w$. Hence, by Proposition \ref{prop1} (b), we know that $\fw([i_w]) = w([i_w]) = [i_w]$. This yields $i_{\fw} \le i_w$ due to the minimality of $i_w$.
\end{proof}

\begin{lem}
    \label{LtoL}
   If $w \in U_n$, and $\f$ is any homing shuffle, then $\f(w) \in U_n$. 
\end{lem}

\begin{proof}
Let $w \in U_n$, and $\f$ an arbitrary homing shuffle.
Lemma \ref{i-decreases} says $i_{\fw} \le i_w$, so that $i_w \ge i_{\fw}$.
Since $w \in U_n$, there exists an $i > i_w$ such that $w(i) \ne i$.
This $i$ also satisfies $i > i_w \ge i_{\fw}$.
Thus, there exists an $i > i_{\fw}$ such that $\fw(i) \ne i$.
In other words, $\fw \in U_n$.
\end{proof}

\begin{thm}
    \label{thm.Un-hopeless}
    If $ w \in U_{n} $ then there is no homing shuffle that sorts $w$.
\end{thm}

\begin{proof}
Let $\f$ be a homing shuffle and $w \in U_n$. Then Lemma \ref{LtoL}, used iteratively, shows that $\fw^m \in U_n$ for all $m \in \NN$. Thus to show that $w$ is not sorted by any homing shuffle it suffices to show $\varepsilon \notin U_n$. But this is obvious, since any $w \in U_n$ must have a non-fixed point.
\end{proof}

The converse is also true, and is a consequence of Theorem \ref{mshufSorts} below.

We now turn to counting the unsortable permutations. We can notice that almost all reducible permutations are unsortable. In fact, a reducible permutation $w \in S_n$ is sortable only if it fixes all $i > i_w$, that is, satisfies $[w(i_w+1),w(i_w+2),\dots,w(n)] = [i_w+1,i_w+2,\dots,n]$.
Hence, any irreducible permutation $u \in I_k \subseteq S_k$ for $k \leq n$ can be extended uniquely to a sortable permutation in $S_n$ by appending the numbers $k+1, k+2, \ldots, n$ at the end of its one-line notation; but extending it in any other way will produce an unsortable permutation.
This helps us compute the number of all unsortable permutations in $S_n$:

\begin{thm}[Size of the unsortables]
    \label{cardUn}
    Let $ I_{n} $ be as in Proposition \ref{prop5}, and $ U_{n} $ be the set of unsortables in $ S_{n} $. Then
    \[
        \card{U_{n}} = n! - \sum_{k=1}^{n} \card{I_{k}}
    \]
\end{thm}
The first few values of $\card{U_n}$ can be seen in Table \hyperref[Un_1-12]{1}. 

\begin{proof}[of Theorem \ref{cardUn}]
We say that a permutation $w \in S_n$ is \emph{sortable} if it satisfies $w\tup{i} = i$ for all $i > i_w$.
Thus, a permutation $w \in S_n$ is unsortable if and only if it is not sortable.
This allows us to rewrite $\card{U_n}$ (that is, the number of unsortable permutations in $S_n$) as $n!$ minus the number of sortable permutations in $S_n$.
But
\begin{align}
    &\tup{\text{number of sortable permutations in $S_n$}} \nonumber\\
    =& \sum_{k=1}^n \tup{\text{number of sortable permutations $w \in S_n$ such that $i_w = k$}}.
    \label{eq.sortables-count.1}
\end{align}

Let us now fix a $k \in [n]$, and count the sortable permutations $w \in S_n$ such that $i_w = k$.
These permutations $w$ have the property that the restriction $ w|_{[k]} $ is an irreducible permutation in $ S_{k} $ (since $i_w = k$ and thus $w([k]) = w([i_w]) = [i_w] = [k]$, whereas $w([i]) \neq [i]$ for all $0 < i < k$), whereas the restriction $ w|_{[k+1,n]} $ is the identity (since otherwise, $w$ would be unsortable).
Conversely, any permutation $w \in S_n$ with this property is sortable and satisfies $i_w = k$.
Thus, we obtain a bijection
\begin{align*}
    \set{\text{sortable permutations $w \in S_n$ such that $i_w = k$}}
    &\to
    I_k , \\
    w &\mapsto w|_{[k]}.
\end{align*}
Consequently, the number of sortable permutations $w \in S_n$ such that $i_w = k$ equals $\card{I_k}$.

Having shown this for all $k \in [n]$, we can now rewrite \eqref{eq.sortables-count.1} as
\begin{align}
    \tup{\text{number of sortable permutations in $S_n$}}
    = \sum_{k=1}^n \card{I_k}.
    \label{eq.sortable-count.2}
\end{align}

But we have seen that the unsortable permutations are just the permutations that are not sortable. Hence, their number is $n!$ minus the number of sortable permutations.
By \eqref{eq.sortable-count.2}, this is $n! - \sum_{k=1}^n \card{I_k}$.
In other words, $\card{U_n} = n! - \sum_{k=1}^n \card{I_k}$.

\end{proof}

\begin{table}[!h]
\label{Un_1-12}
\begin{center}
\begin{tabular}{|c|c|c|c|c|c|c|c|c|c|c|c|c|}
     \hline
     $n$ &  1 & 2 & 3 & 4 & 5 & 6 & 7 & 8 & 9 & 10 & 11 & 12 \\
     \hline 
     $\card{U_n}$ & 0 & 0 & 1 & 6 & 31 & 170 & 1043 & 7230 & 56447 & 493042 & 4782139 & 51122982\\
     \hline
\end{tabular}
\caption{Values of $\card{U_n}$ for $n=1,2,\dots,12$}
\end{center}
\end{table}

\section{Max Shuffles}
Having found a set of permutations that is not sortable by any homing shuffle, we may wonder if this is the minimal such set.
In other words, is there a homing shuffle that sorts all $w \notin U_n$? In this section we answer in the affirmative that there is a family of such homing shuffles.
\begin{defn}[Max Shuffle]
    \label{mShuf}
   We say a homing shuffle $ \f $ is a \emph{max shuffle} if for every $w \in S_n$ with $w\tup{1} \neq 1$, we have
   \[
   \fw\tup{ 1 } = \max \tup{ w\tup{ [2,k]}},
   \qquad \text{where $k := w\tup{1}$}.
   \]
   Informally, after placing the $ k $-th card in its natural position, a max shuffle always places the largest of the cards $k$ has been moved past at the top of the deck.
\end{defn}

The definition of a max shuffle $\f$ uniquely determines how $\f$ transforms a permutation $w$ as long as $w\tup{1}$ is $1$, $2$ or $3$, but leaves some freedom if $w\tup{1} > 3$. Thus, there are many max shuffles for each $n$.

\subsection{Max Shuffles Sort Everything}

We seek to establish the following claim.
\begin{thm}
    \label{mshufSorts}
    If $ \M $ is a max shuffle and $ w \not\in U_{n} $, then $ \M $ sorts $ w $.
\end{thm}
Towards a proof of Theorem \ref{mshufSorts}, we establish a few lemmas. The first of these describes how $i_w$ must change under a max shuffle. 
\begin{lem}
    \label{imw-lem}
    Consider a max shuffle $\M$, and a permutation $w \in S_n$ with $k := w(1)$. Let $i_w$ be as in Definition \ref{irredPerms}. Then we have the following.
    \begin{enumerate}[]
        \item\textbf{(a)} We have $i_w \ge k$.
        \item\textbf{(b)} If $i_w > k$ then $i_{\Mw} = i_w$.
        \item\textbf{(c)} If $i_w = k$ then $i_{\Mw} = k - 1$.
    \end{enumerate}
\end{lem}

\begin{proof}
(a) We note that $k = w(1) \in w([i_w]) = [i_w]$, thus $k \le i_w$, and this proves (a).

Applying (a) to $\Mw$ instead of $w$, we find
\begin{align}
i_{\Mw} &\ge \Mw(1) = \max\tup{w([2,k])}
\qquad \left(\text{by the definition of a max shuffle}\right) \nonumber\\
&\ge k-1, \label{eq.imw-lem.1}
\end{align}
since a maximum of $k-1$ distinct positive integers is $\geq k-1$.
\medskip

(b) Assume $i_w > k$.
By Proposition \ref{prop1} (b), we then have $\Mw([i_w]) = w([i_w]) = [i_w]$.

Hence, it remains to show that $\Mw([i]) \ne [i] $ for all positive $i < i_w$. We show this by considering the three cases $i < k-1$, $i = k-1$ and $i \geq k$ separately:
Fix a positive $i < i_w$. We must show that $\Mw([i]) \ne [i] $.
\begin{enumerate}
\item For $i < k-1$, it follows from \eqref{eq.imw-lem.1} and the definition of $i_{\Mw}$.
\item For $i = k-1$, we observe that Proposition \ref{prop1} (c) yields $\Mw([k-1]) = w([k]) \setminus \set{k} \ne [k-1]$ (because $i_w > k$ entails $w([k]) \ne [k]$). Thus, $\Mw([i]) \ne [i]$ for $i = k-1$.
\item For $i \geq k$,  Proposition \ref{prop1} (b) yields $\Mw([i]) = w([i]) \ne [i]$, since $i < i_w$.
\end{enumerate}
Thus, $\Mw([i]) \ne [i] $ is always verified, and the proof of (b) is complete.
\medskip

(c) Now assume $i_w = k$. Then $w([k]) = [k]$. But Proposition \ref{prop1} (c) shows that
\[
\Mw([k-1]) = w([k]) \setminus \set{k} = [k]\setminus \set{k} = [k-1].
\]
Then this and equation \eqref{eq.imw-lem.1} entails $i_{\Mw} = k-1$, and (c) is proved.
\end{proof}

\begin{lem}
    \label{nonUn-lem}
    Let $\M$ be a max shuffle. Let $w \in S_n$ be such that $w \notin U_n$. Then, $\M\tup{w} \notin U_n$ as well.
\end{lem}

\begin{proof}
Let $k := w(1)$. We first note that $i_w \ge k$ by Lemma~\ref{imw-lem} (a). Hence, for all $i > i_w$, we have $i > i_w \ge k$ and therefore, by the definition of a homing shuffle,
\begin{equation}
    \Mw(i) = w(i) = i
    \label{eq.nonUn-lem.1}
\end{equation}
since $w \notin U_n$ and $i > i_w$. We must show $\Mw \notin U_n$. Equivalently, we must show that $\Mw$ fixes all $i > i_{\Mw}$.
Since $i_w \ge k$, it suffices to consider the cases $i_w > k$ and $i_w = k$.
In the former case, Lemma \ref{imw-lem} (b) yields $i_{\Mw} = i_w$, so that the claim follows from \eqref{eq.imw-lem.1}.
In the latter case, Lemma \ref{imw-lem} (c) yields $i_{\Mw} = k-1$, so that the only $i > i_{\Mw}$ that does not satisfy $i > i_w$ is $k$. Then, by equation \eqref{eq.nonUn-lem.1}, we need only check that $k$ is a fixed point of $\Mw$, but this follows from the definition of a homing shuffle (Definition \ref{shufDef} (a)).
\end{proof} 

Now the claim of Theorem \ref{mshufSorts} follows easily from the above lemmas.

\begin{proof}[of Theorem \ref{mshufSorts}]
    Assume $w \notin U_n$. Then by Theorem \ref{f-terminates} there is a $m \in \NN$ so that $\Mw^m(1) = 1$.
    Thus, $i_{\Mw^m} = 1$. By Lemma \ref{nonUn-lem} (applied iteratively), we know that $\Mw^m \notin U_n$.
    So, from the definition of the unsortables  (Definition \ref{def5}), we have $\Mw^m(i) = i$ for all $i > i_{\Mw^m}$. Since $i_{\Mw^m} = 1$, this means that $\Mw^m$ fixes all $i > 1$ and thus fixes $1$ as well. In other words, $\Mw^m = \varepsilon$.
\end{proof}

Combining Theorem~\ref{mshufSorts} with Theorem~\ref{thm.Un-hopeless}, we obtain the minimality of the set $U_n$.
\begin{cor}
    Let $w \in S_n$. There is no homing shuffle that sorts $w$ if and only if $w \in U_n$.
\end{cor}

\subsection{The Termination Number of a Max Shuffle}

We now consider the termination numbers of the max shuffles. Perhaps unsurprisingly, all max shuffles have the same termination number. 
\begin{thm}
    \label{mshufSt2}
    All max shuffles have the same termination number: that is,
    \[
    \tn \M = \tn \Mo \qquad \text{ for any max shuffles } \M,\Mo .
    \]
\end{thm}

This is essentially because the only freedom a max shuffle has is to permute the cards between the fixed point and the largest card placed at the front, and these transformations do not change what the next front card will be. A precise proof is given below.

We can furthermore determine the shared termination number exactly.

\begin{thm}
    \label{mshufTN}
    Let $\M $ be a max shuffle. Assume that $n \ge 2$. Then 
    \[
    \tn \M = 2n-3.
    \]
\end{thm}

We state our result for $n \ge 2$, since the formula would be false in the (obvious) cases $n < 2$.

We begin our argument for Theorem \ref{mshufSt2} with a claim slightly stronger then our vague statement above. We first define a useful equivalence relation.

\begin{defn}
    \label{shuf_eq}
    Let $w,w' \in S_n$. We write $w \sim w'$ if $w(1) = w'(1)$ and $w(i) = w'(i)$ for all $i > w(1)$.
\end{defn}

It is easy to see that $\sim$ is an equivalence relation on $S_n$. It is also easy to enumerate the equivalence classes. Let $w \in S_n$ be a permutation with $k := w(1)$, and let $[w] := \set{w' \in S_n : w \sim w'}$ be its equivalence class.
Then every $w' \in [w]$ is obtained from $w$ by permuting the underlined portion of its one-line notation in the following equality:
    \[
    w = \tup{k,\underline{w(2),w(3),\dots,w(k)},w(k+1),\dots,w(n)} .
    \]
Thus $\card{[w]} = \tup{k-1}!$. It is not much harder to count the equivalence classes. Indeed, each class $[w]$ is determined by the value $k = w\tup{1}$ and by the $n-k$ values $w(k+1), w(k+2), \ldots, w(n)$. There are $\tup{n-1}\tup{n-2}\cdots k = \dbinom{n-1}{n-k} \cdot \tup{n-k}!$ possible choices for the latter $n-k$ values (as they have to be $n-k$ distinct elements of the set $\set{1,2,\ldots,n}$, but cannot include $k = w\tup{1}$) and thus $\binom{n-1}{n-k} \cdot \tup{n-k}!$ many equivalence classes with the given $w\tup{1} = k$ value. Thus
   \begin{align*}
   \tup{\text{ total number of equivalence classes }} &= \sum_{k=1}^n \binom{n-1}{n-k}\tup{n-k}! = \sum_{k=1}^n \frac{\tup{n-1}!}{\tup{k-1}!}  \\
   &\le \tup{n-1}! \cdot e \qquad \text{with } e = \exp 1.
   \end{align*}
   
The next lemma shows that max shuffles preserve this equivalence relation.

\begin{lem}
    \label{shuf_inv}
    Let $w,w' \in S_n$, and let $\M,\Mo$ be two (possibly the same) max shuffles. 
    
    If $w \sim w'$ then $\M\tup{w} \sim \Mo\tup{w'}$.
\end{lem}

\begin{proof}
Assume $w \sim w'$.
Let $k = w\tup{1} = w'\tup{1}$.
Then, since $w\tup{i}=w'\tup{i}$ for all $i > k$, we have that $w\tup{[k+1,n]} = w'\tup{[k+1,n]}$, and by taking complements we also have that $w\tup{[k]}=w'\tup{[k]}$.
Removing the equal values $w(1) = w'(1)$ from these two sets, we obtain
$w\tup{[2,k]} = w'\tup{[2,k]}$.
Since $\M$ and $\Mo$ are max shuffles, we have
\[
\Mw\tup{1} = \max\tup{w\tup{[2,k]}} = \max\tup{w'\tup{[2,k]}} = \Mowp\tup{1}.
\]
It remains to show that
\begin{equation}
    \Mw\tup{i} = \Mowp\tup{i} \qquad \text{for all } i > \Mw\tup{1}.
    \label{pf.shuf_inv.3}
\end{equation}
We first note that $\Mw\tup{1} = \Mowp\tup{1} \ge k-1$, since the maximum of a set of $k-1$ distinct positive integers is always $\geq k-1$.
Now let $i > \Mw\tup{1}$.
If $i > k$, then the definition of a max shuffle shows that $\Mw\tup{i} = w\tup{i}$ and $ \Mowp\tup{i} = w'\tup{i}$. Combining this with $w\tup{i} = w'\tup{i} $ (which is because $w \sim w'$), we obtain
\[
\Mw\tup{i} = w\tup{i} = w'\tup{i} = \Mowp\tup{i} .
\]
Thus, \eqref{pf.shuf_inv.3} is proved for all $i > k$. But \eqref{pf.shuf_inv.3} also holds for $i = k$ (since the definition of a max shuffle shows that $\Mw\tup{k} = k$ and $\Mowp\tup{k} = k$).
Thus, \eqref{pf.shuf_inv.3} holds for all $i \geq k$, and therefore for all $i > \Mw\tup{1}$ (since $i > \Mw\tup{1} \ge k-1 $ implies $i \ge k$).
This completes our proof.
\end{proof}

From the previous lemmas, it follows that the front card is the same no matter the max shuffle being considered.
\begin{cor}
    \label{f_match}
    Let $ \M,\Mo $ be two max shuffles. Then for all $i \in \mathbb{N} $ and $w \in S_n$, we have
    \[
    \Mw^{i}\tup{ 1 } = \Mow^{i}\tup{ 1 }.
    \]
\end{cor}
\begin{proof}
Since $w \sim w$, repeated application of Lemma \ref{shuf_inv} yields $\Mw^{i} \sim \Mow^{i}.$ Then the definition of $\sim$ (Definition \ref{shuf_eq}) yields $\Mw^{i}\tup{ 1 } = \Mow^{i}\tup{ 1 }$. 
\end{proof}

Now the claim of Theorem \ref{mshufSt2} follows easily from the above.

\begin{proof}[of Theorem \ref{mshufSt2}]
    Let $\M,\Mo$ be two max shuffles. Corollary \ref{f_match} shows that $\Mw^{\tn \Mo}(1) = \Mow^{\tn \Mo}(1) = 1$ for all $w \in S_n$. In other words, $\Mw$ terminates after $\tn \Mo$ iterations on each $w \in S_n$.    
    Thus $\tn \Mo \ge \tn \M$. Similarly, $\tn \M \ge \tn \Mo$. Thus $\tn \M = \tn \Mo$.
\end{proof}

We now seek to establish Theorem \ref{mshufTN}. Before proving this result we show a few lemmas. In all of them, $w$ denotes an arbitrary permutation in $S_n$.

\begin{lem}
\label{max-term-lem1}
    If $w(1) = \Mw(1)$, then $w(1)=1$ and $\Mw = w$.
    That is, if the front card does not change, then $\M$ has already terminated (\textit{i.e}., terminates after $0$ steps).
\end{lem}

\begin{proof}
    Since $w(1)$ is a fixed point of $\Mw$ (Definition \ref{shufDef} (a)) we have that
    \[
    \Mw(w(1)) = w(1) = \Mw(1).
    \]
    Then, as $\Mw$ is a bijection, we may apply $\Mw^{-1}$ to both sides of the above. This yields $w(1)=1$.  Thus, we have terminated and $\Mw = w$.
\end{proof}

\begin{lem}
    \label{max-term-lem2}
    Let $ \M $ be a max shuffle, and $ k = w\tup{ 1 }$ be the front card.
    \[
        \text{If } w\tup{1} \ge \Mw\tup{ 1 } \text{ then } \Mw^{k-1}\tup{ 1 } = 1 .
    \]
    In other words, if $\M$ decreases the front card of $w$, then $ \M $ terminates on $ w $ after $ k-1 $ iterations.
\end{lem}
\begin{proof}
We proceed by induction on $k$. The base case $k=1$ follows from Lemma \ref{max-term-lem1} (since $\Mw\tup{1} \le w\tup{1} = k = 1$ entails $\Mw\tup{1} = 1 = w\tup{1}$).
Now consider some $k > 1$. Assume for our inductive hypothesis that the statement holds for $k-1$, and that $k := w\tup{1} \ge \Mw\tup{1}$. We wish to show $\Mw^{k-1}\tup{1}=1$. We first note that if $w\tup{1}=\Mw\tup{1}$ then by Lemma \ref{max-term-lem1} we have $w\tup{1}=1$, contradicting $w\tup{1} = k > 1$.
Thus we in fact have $w\tup{1} > \Mw\tup{1}$. Now by the definition of max shuffle (Definition \ref{mShuf})
\begin{align*}
\Mw\tup{1} = \max\tup{w\tup{[2,k]}} \ge k-1,
\end{align*}
since the maximum of $k-1$ distinct positive integers is always $\ge k-1$.
Since $\Mw\tup{1} < w\tup{1} = k$, this yields $\Mw\tup{1}=k-1$. Thus $\max\tup{w\tup{[2,k]}} = \Mw\tup{1} = k-1$, and therefore $w\tup{[2,k]}=[k-1]$, since the maximum of $k-1$ distinct positive integers is $k-1$ if and only if they are simply the integers $1,2,\dots,k-1$.
Taking the union of this with $\set{w\tup{1}} = \set{k}$, we conclude that $\set{w\tup{1}} \cup w\tup{[2,k]} = \set{k} \cup [k-1]$. In other words,
$w\tup{[k]} = [k]$.

Next we shall show that $\Mw^2\tup{1} \le \Mw\tup{1}$. Indeed, if $\Mw\tup{1} = 1$, then this is clear (since in this case $\Mw^2 = \M_{\Mw} = \Mw$). Otherwise, it follows from
\begin{align*}
\Mw^{2}\tup{1} &= \max\tup{\Mw\tup{[2,k-1]}} \qquad\tup{\text{by Definition \ref{mShuf}}}\\
&\le \max\tup{\Mw\tup{[k-1]}} \qquad\tup{\text{since } \Mw\tup{[2,k-1]} \subseteq \Mw\tup{[k-1]}}\\
&= \max\tup{w\tup{[k]}\setminus \set{k}} \qquad\tup{\text{by Proposition \ref{prop1} (c)}} \\
&= \max\tup{[k]\setminus \set{k}} \qquad\tup{\text{since } w\tup{[k]}=[k]} \\
& = k-1 = \Mw\tup{1}.
\end{align*}

In other words, $k-1 = \Mw\tup{1} \ge \Mw^2\tup{1}$.
Hence, we may apply our induction hypothesis to $\Mw$ and $k-1$ to conclude that $\M_{\Mw}^{k-2}\tup{1} = 1$.
Since $\M^{k-2}_{\Mw}=\M^{k-2}\tup{\Mw}=\Mw^{k-1}$, this proves our claim.

\end{proof}

\begin{cor}
    \label{max-term-cor}
    Let $ \M $ be a max shuffle. 
    \[
        \text{If } w\tup{1} \ge \Mw\tup{ 1 } \text{ then } \Mw^{n-1}\tup{ 1 } = 1 .
    \]
\end{cor}
\begin{proof}
    Apply Lemma \ref{max-term-lem2} to $k = w\tup{1}$, and notice that $k \le n$.
\end{proof}

Now we may easily prove Theorem \ref{mshufTN}.

\begin{proof}[of Theorem \ref{mshufTN}]
Let $w \in S_n$. We shall show that $\M$ terminates after $2n-3$ iterations on $w$. This will yield $\tn \M \le 2n-3$, thus proving one half of the theorem.

The sequence of front cards
\[
\tup{\Mw^0(1),\ \Mw^1(1),\ \Mw^2(1),\ \dots}
\]
cannot increase forever.
Thus there must be a $d \ge 0$ such that $\Mw^d(1) \ge \Mw^{d+1}(1)$.
Consider the smallest such $d$.
Then the sequence of front cards strictly increases for the first $d$ iterations:
\[
k := w(1) = \Mw^0(1) < \Mw^1(1) < \Mw^2(1) < \cdots < \Mw^d(1).
\]
This entails $d \le n-k$, since the longest increasing sequence of integers starting at $k$ and bounded above by $n$ is the sequence $k,k+1,\dots,n$ of length $n-k+1$. 

Now, if $k=1$, then $\M$ terminates after $0$ iterations on $w$, and hence after $2n-3$ iterations too.
Hence, assume $k \ge 2$. Then $d \le n-k \le n-2$. Now $\Mw^d(1) \ge \Mw^{d+1}(1)$ allows us to apply Corollary \ref{max-term-cor} to $\Mw^d$. This yields $\M^{n-1}_{\M^d_w}\tup{1} = 1$. In other words, $\Mw^{d+n-1}\tup{1} = 1$.
Hence, $\M$ terminates after $d+n-1$ iterations on $w$, and thus after $2n-3$ iterations too (since $d \le n-2$ and so $d+n-1 \le 2n-3$).
This proves $\tn \M \le 2n-3$.

It remains to show $\tn \M \ge 2n-3$.
By Theorem \ref{mshufSt2}, it suffices to show this for one given max shuffle $\M$.
Let $w \in S_n$ be the permutation with one-line notation $(2,3,4,\dots,n,1)$. For example, if $n=5$, then $\oln w = (2,3,4,5,1)$.
We may consider the max shuffle $\M$ described explicitly by
\[
\M\tup{w} = \cyc{w\tup{k},k,k'} w, \qquad \text{ where } k = w\tup{1} \text{ and } k' = \max\tup{w\tup{[2,k]}}
\]
(as long as $k \neq 1$; otherwise $\M\tup{w} = w$; we furthermore understand $\cyc{w\tup{k},k,k'}$ to mean $\cyc{w\tup{k},k}$ when $k' = w\tup{k}$).
It is not hard to see that $\M$ is really a max shuffle.
Informally, $\M$ first swaps the front card $k$ with $w(k)$ before swapping $w(k)$ with largest card $k'$.
Let us use the notation $w \overset{\M}{ \mapsto } w'$ as shorthand for $w' = \M(w)$.
Then we may describe how $\M$ acts on $w=\tup{2,3,4,5,1}$ for $n = 5$ as follows. In following computations we place a dot over $k$, a hat over $w(k)$, and a tilde over $k'$.

\begin{align*}
    \tup{\dot{2},\hat{3},4,5,1} &\overset{\M}{ \mapsto } \tup{\dot{3},2,\hat{4},5,1}\\
     &\overset{\M}{ \mapsto } \tup{\dot{4},2,3,\hat{5},1}\\
     &\overset{\M}{ \mapsto } \tup{\dot{5},2,3,\tilde{4},\hat{1}}\\
     &\overset{\M}{ \mapsto } \tup{\dot{4},2,\tilde{3},\hat{1},5}\\
     &\overset{\M}{ \mapsto } \tup{\dot{3},\tilde{2},\hat{1},4,5}\\
     &\overset{\M}{ \mapsto } \tup{\dot{2},\hat{1},3,4,5}\\
     &\overset{\M}{ \mapsto } \tup{\dot{1},2,3,4,5}.
\end{align*}
The process carries over to arbitrary $n$. Each card (except $1$) appears in front once prior to the $n$-th card and once after.
We are using the notation $[a,b]$ for the list $\tup{a,a+1,\ldots,b-1,b}$.
 \begin{align*}
        (\dot{2},\hat{3},4,\dots,n,1) &\overset{\M}{ \mapsto } (\dot{3},2,\hat{4},\dots,n,1)\\
                          &\overset{\M}{ \mapsto } (\dot{4},2,3,\dots,n,1)\\
                          & \cdots \\
                          &\overset{\M}{ \mapsto } (\dot{k},[2,k-1],[\hat{k+1},n],1) \quad\tup{\text{for each } k \in \set{2,3,\ldots,n} \text{ in order} }\\
                          & \cdots \\
                          &\overset{\M}{ \mapsto } (\dot{n},2,3,\dots,\tilde{n-1},\hat{1})\\
                          &\overset{\M}{ \mapsto } (\dot{n-1},2,3,\dots,\tilde{n-2},\hat{1},n)\\
                          &\overset{\M}{ \mapsto } (\dot{n-2},2,3,\dots,\hat{1},n-1,n)\\
                          & \cdots \\
                          &\overset{\M}{ \mapsto } (\dot{k},[2,\tilde{k-1}],\hat{1},[k+1,n]) \quad\tup{\text{for each } k \in \set{n,n-1,\ldots,1} \text{ in order}}\\
                          & \cdots \\
                          &\overset{\M}{ \mapsto } (\dot{1},2,3,\dots,n-1,n).
\end{align*}
Thus $\M$ terminates in exactly $2n-3$ iterations on $w$.
This proves $\tn \M \ge 2n-3$ for our max shuffle $\M$, and therefore for every max shuffle $\M$, as desired.
\end{proof}

\section{Concluding discussions}

While the above may be a helpful framework for approaching problems concerning homing shuffles, it does not address the current open questions in this area. We highlight three of these now.

\begin{enumerate}[]
    \item\textbf{Question 1.} How many permutations are sorted by $\mck$?
    \item\textbf{Question 2.} What are the permutations sorted by $\Ts$ and how many are there?
    \item\textbf{Question 3.} What is the termination number of $\Ts$?
\end{enumerate}

In regards to question 1, McKinley characterizes the permutations sorted by $\mck$ in \cite{McK2015}, and shows that their proportion to the number $n!$ of all permutations is bounded above by $\frac{e}{n}$. However she also notes that numerical evidence suggests the proportion is closer to $\frac{1}{n}$. One would hope to either improve the bound or find an explicit formula.

In regards to question 3, only bounds are known. Aside from the trivial $\tn \Ts \le 2^{n-1}$ from Theorem \ref{f-terminates}, Knuth has shown that $\tn \Ts \le f_{n+1} - 1$ \cite{TAoCP4A}. However, these upper bounds are far from the quadratic lower bound shown by Morales and Sudborough \cite{morales2010}. Indeed, numerical evidence suggests that $\tn \Ts$ is much closer to the lower bound \cite{OEIS_TS}. Thus any improvement of the upper bound would represent substantial progress.

In regards to question 2, it is easy to see one characterization, albeit not one that leads to any substantial bound. We claim it follows easily from the definition of $\Ts$ that if $\Ts$ sorts $w \in S_n$ then there are some $k_1,k_2,\dots,k_m \in [n]$ with $m \le n$ such that
\[
w = r_{k_{1}}r_{k_{2}} \cdots r_{k_{m-1}}r_{k_{m}} \text{ and } k_i \text{ is a fixed point of } r_{k_{1}}r_{k_{2}} \cdots r_{k_{i-1}} \text{ for each } 1 < i \le m.
\]
Then since $r_{k_i}$ fixes all $i > k_i$ it follows that any increasing tuple $1 < k_1 < k_2 < \cdots < k_m \le n$ corresponds to a sortable permutation.
Then, as there are $2^{n-1}$ such increasing tuples, we know there are at least $2^{n-1}$ permutations that $\Ts$ sorts (the distinctness of these permutations is easy to check).
Of course, these are not the only sortables, since $r_{2}r_{3}{r_{4}}r_{2}r_{5}r_{4}$ is also sortable in $S_5$. Thus a possible way to obtain a bound is to enumerate all the sequences that correspond to the sortable $w$.

\acknowledgments
I would like to thank Professor Darij Grinberg for advising me in this research. I am deeply appreciative of the many hours spent discussing the problem and teaching me how to write a mathematics paper. I would also like to thank Drexel University for the opportunity to participate in mathematics research for my COOP cycle. We also thank the referees for their cordial and constructive comments, and in particular for bringing our attention to the master's thesis of Joyce Pechersky \cite{Pe2023}.

\bibliographystyle{plain} 
\bibliography{references} 
\end{document}